\documentclass[10pt]{amsart}
\usepackage{amsmath}
\usepackage{amsfonts}
\usepackage{amssymb}
\usepackage{amsthm}
\usepackage{url}
\usepackage{dsfont}
\usepackage{graphicx}
\usepackage{caption}
\usepackage{subcaption}
\usepackage{comment} 
\usepackage{stmaryrd}
\usepackage{hyperref}
\usepackage{todonotes}
\usepackage{color}
\usepackage{enumerate,tikz-cd}

\newcommand{\git}{\mathbin{
  \mathchoice{/\mkern-6mu/}
    {/\mkern-6mu/}
    {/\mkern-5mu/}
    {/\mkern-5mu/}}}

\numberwithin{equation}{section}

\newtheorem{proposition}{Proposition}[section]
\newtheorem{lemma}[proposition]{Lemma}
\newtheorem{theorem}[proposition]{Theorem}

\theoremstyle{definition}
\newtheorem{remark}[proposition]{Remark}
\newtheorem{definition}[proposition]{Definition}
\newtheorem{example}[proposition]{Example}

\DeclareMathOperator{\End}{End}
\DeclareMathOperator{\Id}{Id}
\DeclareMathOperator{\tr}{tr}

\DeclareMathOperator{\Aut}{Aut}

\DeclareMathOperator{\Proj}{Proj}
\DeclareMathOperator{\Hom}{Hom}
\DeclareMathOperator{\Grad}{Grad}
\DeclareMathOperator{\Spec}{Spec}

\DeclareMathOperator{\ch}{ch}

\DeclareMathOperator{\rk}{rk}

\DeclareMathOperator{\Ima}{Im}
\DeclareMathOperator{\Rea}{Re}
\DeclareMathOperator{\Lie}{Lie}

\DeclareMathOperator{\Coh}{Coh}

        \DeclareMathOperator{\uMaps}{\underline{Maps}}
        \DeclareMathOperator{\Charges}{PStab}
        \newcommand{\cP}{{\mathcal P}}
        \newcommand{\cS}{{\mathcal S}}

        \newcommand{\fl}{{\mathfrak l}}

        \newcommand{\abs}[1]{\left|  #1\right| } 

\newcommand{\R}{\mathbb{R}}
\newcommand{\C}{\mathbb{C}}
\newcommand{\Gbb}{\mathbb{G}}

\newcommand{\Z}{\mathbb{Z}}
\newcommand{\Hbb}{\mathbb{H}}
\newcommand{\Q}{\mathbb{Q}}

\newcommand{\pr}{\mathbb{P}}
\renewcommand{\epsilon}{\varepsilon}

\newcommand{\scC}{\mathcal{C}}
\newcommand{\scA}{\mathcal{A}}

\newcommand{\mft}{\mathfrak{t}}

\newcommand{\mfg}{\mathfrak{g}}

\newcommand{\mfk}{\mathfrak{k}}

\renewcommand{\L}{\mathcal{L}}
\newcommand{\X}{\mathcal{X}}
\newcommand{\Y}{\mathcal{Y}}

\newcommand{\scS}{\mathcal{S}}
\newcommand{\ddbar}{\partial\overline{\partial}}
\renewcommand{\phi}{\varphi}
        \newcommand{\cX}{{\mathcal X}}

\title[Stability conditions in geometric invariant theory]{Stability conditions in geometric invariant theory}

\author[Ruadha\'i Dervan, with an appendix by Andr\'es Ib\'a\~nez N\'u\~nez]{Ruadha\'i Dervan  \\ \\ { with an appendix by Andr\'es Ib\'a\~nez N\'u\~nez}}

\address{School of Mathematics and Statistics, University of Glasgow, University Place, Glasgow G12 8QQ, United Kingdom}\email{ruadhai.dervan@glasgow.ac.uk}

\address{Andr\'es Ib\'a\~nez N\'u\~nez, Mathematical Institute, University of Oxford, Andrew Wiles Building, Radcliffe Observatory Quarter (550), Woodstock Road, Oxford
OX2 6GG, United Kingdom} \email{ibaneznunez@maths.ox.ac.uk}

\begin{document}

\begin{abstract} We explain how structures analogous to those appearing in the theory of stability conditions on abelian and triangulated categories arise in geometric invariant theory. This leads to an axiomatic notion of a central charge on a scheme with a group action, and ultimately to a notion of a stability condition on a stack analogous to that on an abelian category.  In an appendix by A.  Ib\'a\~nez N\'u\~nez, it is explained how central charges can be viewed through the graded points of a stack. We use these ideas to introduce an axiomatic notion of a stability condition for polarised schemes, defined in such a way that K-stability is a special case.

In the setting of axiomatic geometric invariant theory  on a smooth projective variety, we produce an analytic counterpart to stability and explain the role of the Kempf-Ness theorem. This clarifies many of the structures involved in the study of deformed Hermitian Yang-Mills connections, $Z$-critical connections and $Z$-critical K\"ahler metrics. \end{abstract}


\maketitle

\section{Introduction}

Mumford's geometric invariant theory gives a method of constructing quotient spaces in algebraic geometry, with many important applications to the construction of moduli spaces \cite{GIT}. These quotients parametrise \emph{polystable} orbits---the unstable orbits are discarded to ensure a separated quotient. 

Perhaps the most powerful outcome of Mumford's work was not geometric invariant theory itself, but rather the introduction of the  notion of \emph{stability}, which has been fundamental to an enormous amount of further work. We mention two examples. The first is the notion of \emph{slope stability} of a coherent sheaf, which  led to Rudakov's abstraction to stability on general \emph{abelian categories} (where there is no possible interpretation via geometric invariant theory) \cite{rudakov} and ultimately to Bridgeland's very general theory of stability conditions on \emph{triangulated categories} \cite{bridgeland}, building on ideas of Douglas motivated by string theory \cite{douglas}. Here geometric invariant theory is used more as a motivational philosophy rather than as a direct tool.

The second example we give is Mumford's construction of the moduli space of (stable) curves \cite{mumford}, and its higher dimensional (and largely conjectural) analogue of \emph{K-stability} of a polarised variety \cite{tian, donaldson-toric}. Here again while Mumford's theory can successfully be applied in the special case of curves, in higher dimensions K-stability is only \emph{modelled on}  geometric invariant theory, which is similarly used as a motivational philosophy. 

Bridgeland's notion of a stability condition is the appropriate one in the presence of some linearity, essentially due to the necessity of the presence of abelian categories, but many interesting problems in algebraic geometry have no such linearity; this is notably the case for the theory of stability of polarised varieties. Roughly speaking, in Bridgeland's theory stability is only defined for certain abelian subcategories $\scA$ of the triangulated category (satisfying various hypotheses), where stability involves choosing a \emph{central charge} $Z: \scA \to \C$ and demanding that for subobjects $S \subset E$  the \emph{phase inequality} $$\arg Z(E) > \arg Z(S)$$ holds (with $\arg$ denoting the argument of a complex number); this type of inequality is precisely of a form analogous to those arising in the traditional theory of slope stability of coherent sheaves. It is then important to assume that central charges are complex valued, as this is a basic step in Bridgeland's proof that the space of stability conditions on a triangulated category forms a complex manifold. Just as Bridgeland's theory gives a general way of understanding stability in the presence of a sort of linearity, it is natural to ask if one can extend the theory beyond the linear setting.

This note revisits some of the foundational ideas of geometric invariant theory, with the aim of developing a general parallel of Bridgeland's work. We accomplish roughly the easier half of this: motivated by a new notion of a central charge associated to a group action on a projective scheme, we introduce a notion of a central charge on a general stack, motivated by Halpern-Leistner's approach to geometric invariant theory on stacks \cite{DHL}. When this stack is the stack of coherent sheaves, we then explain how this relates to the classical notion of a central charge. Our real interest, however, is in the stack of polarised schemes:  here we use this to define an \emph{axiomatic notion of stability for polarised schemes}. As we explain, K-stability is then a special case of $Z$-\emph{stability}, with $Z$ a central charge. An \emph{ad hoc} notion of stability for polarised varieties was introduced in \cite{stabilityconditions}, and our motivation here is to give a more axiomatic approach to much the same notion.

In an appendix by A.  Ib\'a\~nez N\'u\~nez, the relationship with Halpern-Leistner's notion of the stack of graded points of a stack is discussed and explained in detail, and a thorough stack-theoretic treatment of the notion of a central charge is given. These results emphasise that central charges are natural objects associated to stacks. In addition, it is observed there that the space of central charges naturally has the structure of a complex vector space, proving a basic (loosely, abelian) counterpart to the complex-manifold structure of the space of stability conditions on a triangulated category.

We emphasise that this only solves half of the problem, really giving a stacky analogue of Rudakov's notion of stability on an abelian category. To generalise Bridgeland's theory is considerably more challenging---the closest analogy is the requirement to extend from coherent sheaves to \emph{complexes} of coherent sheaves, and it is not clear what the right analogue of a complex should be for more general stacks---especially the stack of polarised schemes. The author expects that the right categorical generalisation of K-stability---parallel to stability conditions on triangulated categories---should involve a larger overlying categorical structure, with stability defined then for appropriate substacks of this larger structure. On these substacks, stability should precisely require choosing a central charge on the stack, with stability then meaning what we introduce in the present work.

While this more categorical generalisation of K-stability is speculative, general notions of stability are of interest even for a fixed polarised variety, where the lack of ``global'' structures analogous to Bridgeland stability conditions should be less problematic. The reason for this interest is in links with geometric partial differential equations and moment maps, as we explain in more detail below. Developing the theory purely algebraically appears to be very challenging, and we leave this for future work; for example---away from the important special case of Fano varieties \cite{xu}---basic questions such as Zariski openness of the stable locus are completely open even for K-stability of polarised varieties. 

\subsection*{Moment maps} Beyond the introduction of the general notion of stability in algebraic geometry, the most powerful outcomes of Mumford's geometric invariant theory arose through resulting links between algebraic and differential geometry. These arise through the \emph{Kempf--Ness theorem} \cite{kempf-ness, kirwan}, which states that \emph{polystability} of an orbit can be characterised through the exisence of zeroes of \emph{moment maps}, as used in \emph{symplectic} geometry. This characterisation then descends to a homeomorphism of the algebraic and symplectic quotients by Kirwan \cite{kirwan}. The notions of slope stability of holomorphic vector bundles and K-stability of smooth polarised varieties---which can be thought of as infinite-dimensional analogues of stability in geometric invariant theory---then have analytic counterparts through the existence of \emph{Hermitian Yang--Mills connections} and \emph{constant scalar curvature K\"ahler metrics}, which can each be viewed as zeroes of infinite-dimensional moment maps \cite{atiyah-bott, fujiki, donaldson-moment}. Thus these moment map equations give differential-geometric approaches to moduli problems in algebraic geometry. 

The second goal of this note is to give---in light of the axiomatic notion of a central charge in geometric invariant theory---a differential-geometric counterpart to central charges and resulting stability conditions, motivated by (and using) the Kempf--Ness theorem on a smooth projective variety. For a central charge $Z$ we define a \emph{complex moment map} and use it to define \emph{$Z$-critical points}, and under a strong ``subsolution'' hypothesis we prove that the existence of $Z$-critical points is equivalent to $Z$-polystability---the axiomatic notion of stability defined using the central charge. 

Infinite-dimensional counterparts of this idea have become increasing prominent in recent years, notably for equations that arose in string theory and mirror symmetry. The most well-understood of these is the \emph{deformed Hermitian Yang-Mills equation}, which is the mirror of the \emph{special Lagrangian equation} under SYZ mirror symmetry \cite{LYZ, MMMS}. The deformed Hermitian Yang-Mills equation roughly corresponds to stability of coherent sheaves with respect to a certain central charge, and for more general central charges one obtains the condition for a connection on a holomorphic vector bundle to be a $Z$\emph{-critical connection}, as introduced with McCarthy and Sektnan \cite{DMS}. A parallel, more challenging analogue for smooth polarised varieties was introduced in \cite{stabilityconditions,momentmaps}, through the theory of $Z$-\emph{critical K\"ahler metrics}, linked with the notion of $Z$\emph{-stability} of the polarised variety \cite{stabilityconditions}. While---with the notable exception of deformed Hermitian Yang--Mills connections on line bundles \cite{chen, collins-jacob-yau, jacob-yau}---it remains an open problem to characterise existence of solutions of these partial differential equations in terms of the associated notions of stability, the present work gives a finite-dimensional analogue of these results which does provide evidence for generalisations of the Kempf--Ness theorem to the aforementioned infinite-dimensional settings. In joint work with Hallam providing a sequel to the present note, a geometric recipe is given which canonically associates  complex moment maps to central charges in the categories of polarised manifolds and holomorphic vector bundles, building on some of the ideas presented here \cite{momentmaps}.

In both the variety and coherent sheaf settings, the $Z$-critical equation can be associated to very general classes of central charges, but basic properties---such as ellipticity of the partial differential equation---cannot hold in complete generality. This is best understood in the (smooth) coherent sheaf setting, where the $Z$-\emph{subsolution} condition governs ellipticity of the equation along with various other geometric properties of the equation \cite{collins-jacob-yau, DMS, takahashi}. This subsolution condition---which as we explain here has a finite dimensional analogue---is a positivity condition, and it seems plausible to suggest that these notions are related to the various axioms of Bridgeland's stability conditions, which themselves can be thought of as positivity conditions (as an explicit example, the existence of Harder--Narasihman filtrations assumed in the definition of a Bridgeland stability condition is a consequence of a kind of convexity, which is in turn a consequence of a kind of positivity). 

 The analytic state-of-the-art in both the variety and coherent sheaves cases is closer to stability on an abelian category than a triangulated one, and it remains an important problem to extend the notion of a $Z$-critical connection to more general complexes. The manifold setting is more mysterious again, but the perspective presented here, along with the parallels between the coherent sheaf and variety theories, provides  analytic motivation for the appearance of an analogue of the derived category of coherent sheaves in the setting of polarised schemes. That is, given that one should expect the need to consider the triangulated setting and categorical obstructions to fully understand the existence of $Z$-critical connections in relation to Bridgeland stability, to understand the existence of $Z$-critical K\"ahler metrics one should also need to understand deeper categorical obstructions. Optimistically one may expect to be able to actually \emph{obtain} the right categorical extensions from a deep understanding of the analysis underlying the $Z$-critical K\"ahler condition, much as one would hope to be able to recover the notion of Bridgeland stability through understanding the analysis underlying the $Z$-critical connection condition. We remark also that there is evidence that Bridgeland stability may not be exactly the right condition to be equivalent to the existence of deformed Hermitian Yang--Mills connections \cite{collins-shi, jacob-sheu}, with this philosophy then suggesting that one may be able to obtain a similar algebro-geometric stability condition through the analysis underlying the deformed Hermitian Yang--Mills condition: the stability condition should be the one that is equivalent to existence of solutions of this equation.

In any case, our work here gives a finite-dimensional motivation for these geometric partial differential equations, and in particular gives a dictionary for passing from standard moment map constructions---such as geometric flows---to interesting infinite-dimensional ones. We thus emphasise that the main purpose of this note is to use geometric invariant theory to understand the right general axiomatic notion of stability for polarised schemes---and for general stacks---and to highlight how this viewpoint can be used to motivate the appearance of differential-geometric counterparts to abstract stability conditions. There are relatively few proofs,  with the focus instead being on providing definitions which we hope will motivate further work. We also hope that that these ideas clarify many of the basic structures appearing in recent work on general relationships between abstract stability conditions and geometric partial differential equations.

\subsection*{Remark}
Independent forthcoming work of Haiden--Katzarkov--Kontsevich--Pandit aims to develop an analytic counterpart to Bridgeland stability for general triangulated categories; they title their programme ``categorical K\"ahler geometry'' (see \cite{HKKP1, HKKP2} for precursors). Their programme has a substantial overlap with the ideas developed here, and indeed the $Z$-critical condition used here bears some similarity to equations that appear in their work. The author thanks F. Haiden and P. Pandit for discussions on these ideas.

\subsection*{Outline}
We discuss stability in geometric invariant theory in Section \ref{classical-GIT}, before turning to the axiomatic notion of $Z$-stability for groups actions on schemes in Section \ref{axiomatic} and on stacks in Section \ref{stacks}. Section \ref{examples} then explains the special cases of stability conditions on the stacks of coherent sheaves and polarised schemes. The analytic counterpart to $Z$-stability is described in Section \ref{sec:momentmaps} through the theory of complex moment maps, and we end with a discussion of various natural analytic structures in complex moment map theory in Section \ref{structures}.

\subsection*{Acknowledgements} I thank  Andr\'es Ib\'a\~nez N\'u\~nez, Frances Kirwan and John McCarthy for many helpful discussions on this topic and for their interest in these ideas, and I further thank Andr\'es Ib\'a\~nez N\'u\~nez for providing the appendix to this work. In addition I thank S\'ebastien Boucksom, Michael Hallam, Eloise Hamilton, Eveline Legendre and Lars Sektnan for their comments, Frances Kirwan and Xiaowei Wang for the suggestion to incorporate stacks and the referee for their helpful suggestions. I was funded by a Royal Society University Research Fellowship for the duration of this work (URF\textbackslash R1\textbackslash 201041).

\section{Stability conditions in geometric invariant theory} 

\subsection{The classical theory}\label{classical-GIT}
Let  $(X, L)$ be a polarised projective scheme, and let $G$ be a reductive group acting on $(X,L)$. We briefly recall some of the basic ideas of geometric invariant theory (GIT), which produces a quotient $X \git G$ of $X$ by $G$. The original reference is Mumford's book \cite{GIT}, and good surveys are given by Hashimoto \cite{hashimoto} and Thomas \cite{thomas}. Rather than parametrising all orbits, the quotient represents only \emph{polystable} orbits, which we now define.

Let $\lambda: \C^* \hookrightarrow G$ be a one-parameter subgroup. For a given $x \in X$, we call $$y = \lim_{t\to 0} \lambda(t).x$$ the \emph{specialisation} of $x$ under $\lambda$; we also say that $x$ \emph{degenerates} to $y$ under $\lambda$. As $y$ is fixed by the $\C^*$-action, there is a $\C^*$-action on the one-dimensional complex vector space $L_y$; this action is multiplication by $t^{\nu(y, \lambda)}$ for an integer $\nu(y, \lambda)\in \Z$ called the \emph{weight}. We think of the weight as an assignment $$(y, \lambda) \to \nu(y, \lambda) \in \Z,$$ where $\lambda: \C^* \hookrightarrow G_y$ is a one-parameter subgroup of the stabiliser $G_y$ of $y$.

\begin{definition} We say that $x \in X$ is
\begin{enumerate}[(i)]
\item \emph{semistable} if for all one-parameter subgroups $\lambda$ of $G$, we have ${\nu(y, \lambda)} \geq 0$;
\item \emph{polystable} if for all one-parameter subgroups $\lambda$ of $G$, we have ${\nu(y, \lambda)} \geq 0$, with equality if and only if $ \lim_{t\to 0} \lambda(t).x \in G.x$; 
\item \emph{stable} if for all non-trivial one-parameter subgroups $\lambda$ of $G$, we have ${\nu(y, \lambda)} > 0$;
\item \emph{unstable} otherwise.
\end{enumerate}
Here in each case $y$ is the specialisation of $x$ under $\lambda$. These are conditions on the orbit  $G\cdot x$, so we also that the orbit $G\cdot x$ is semistable, polystable or stable respectively.
\end{definition}

The general machinery of GIT produces a quotient space $X \git G$ which parametrises \emph{polystable orbits}. The first important step in this process is that both the stable locus $X^s$ and the semistable locus $X^{ss}$ are Zariski open. The next step is to prove the existence of a surjective morphism $X^{ss} \to X\git G$ such that distinct polystable orbits are mapped to distinct points in $X\git G$. That this map is well-defined is equivalent to the fact that the closure of each semistable orbit contains a \emph{unique} polystable orbit.

\begin{remark}This is not how stability in GIT is usually presented; rather, it is an equivalent characterisation provided by the Hilbert-Mumford criterion. More typically GIT is presented instead using invariant global sections of $L^k$, from which constructing the GIT quotient is almost tautological via the $\Proj$ construction. The Hilbert--Mumford then gives a way of geometrically interpreting polystable orbits \cite[Section 2]{GIT}.
 \end{remark}

 It will be useful to relate one-parameter subgroups to associated elements of the Lie algebra $\mfg$ of $G$. Consider first a maximal (complex) torus $T^{\C} \subset G$ with Lie algebra $\mft^{\C}$. The \emph{cocharacter lattice} has points consisting of the kernel of the exponential map $\exp: \mft^{\C} \to T^{\C}$; these are in bijection with the one-parameter subgroups $\C^*\to T^{\C}$. Splitting  $\mft^{\C} = \mft \oplus i\mft$, the $\R$-span of the cocharacter lattice is the real Lie algebra $\mft$ of the torus $T$, and the $\C$-span is the Lie algebra $\mft^{\C}$ itself. We denote by $\mft_{\Q} \subset \mft$ the $\Q$-span of the cocharacter lattice, so that the $\Q$-vector space $\mft_{\Q}\oplus i\mft_{\Q} \subset \mft^{\C}$ is a dense subset. We call the spaces $\mft_{\Q}$ and $\mft_{\Q}\oplus i\mft_{\Q}$ the collection of \emph{rational points} of $\mft$ and $\mft^{\C}$ respectively.

We next turn to the Lie algebra $\mfg$, where we say that $u\in\mfg$ is \emph{rational} if there exists a maximal torus $T^{\C}$ of $G$ such that $u \in \mft^{\C} \subset \mfg$ is a rational point. We denote by $\mfg_{\Q}$ the set of rational points of $\mfg$, and note that while this is not a $\Q$-vector space in general, $\mfg_{\Q}$ is nevertheless a dense subset of $\mfg$. Similarly, writing $G$ as the complexification of a maximal compact subgroup $K \subset G$, the Lie algebra $\mfg$ splits as $$\mfg \cong \mfk \oplus i\mfk,$$ where $\mfk = \Lie K$. With an analogous definition of $\mfk_{\Q}$, the $\R$-span of $\mfk_{\Q}$ is $\mfk$ itself, meaning $\mfk_{\Q} \subset \mfk$ is again dense.
 
As the stabiliser $G_x$ of $x \in X$ is not  reductive in general, it is not the case that $\mfg_{x,\Q}$ as defined above is dense in $\mfg_{x}$ for all $x\in X$. Instead we fix a maximal compact subgroup $K_x \subset G_x$, and let $K_x^{\C} \subset G_x$ denote its complexification. We then repeat the above discussion for the reductive group $K_x^{\C}$, and denote by  $ \mfk_{x,\Q}$ and  $\mfk^{\C}_{x,\Q}$ the rational points of $ \mfk_{x}$ and $ \mfk_{x}^{\C}$ respectively.

\begin{lemma}[Additivity] The weight function extends to a Lie algebra character   $$\nu(x, \cdot): \mfk_{x} \to \R.$$ \end{lemma} 

Similarly the weight function extends to a \emph{complex}-valued Lie algebra character  $\nu(x, \cdot): \mfk^{\C}_{x} \to \C.$ That the weight is additive on the cocharacter lattice implies that it extends linearly to a linear function  $\mfk_{x} \to \R$. Denoting by $$g\cdot \lambda = g\circ \lambda \circ g^{-1}$$ the adjoint action of $G$ on one-parameter subgroups, associated to a one-parameter subgroup $\lambda \hookrightarrow G_x$, there is then an induced one-parameter subgroup $g\cdot \lambda \hookrightarrow G_{g(x)}.$ The conjugation-invariance property $$\nu(g(x), \lambda) = \nu(x, g\cdot \lambda)$$  implies from this that $\nu$ is actually a Lie algebra character, namely that  $\nu(x, [\cdot, \cdot]) = 0$.

We also recall an equivariance property that is more global in nature. Consider a connected subscheme $B \subset X$ such that $\lambda$ is a one-parameter subgroup contained in the stabiliser $G_x$ of $x$ for all (closed) points $x \in B$.
 
\begin{lemma}[Equivariant constancy] The value $\nu(x, \lambda)$ is independent of $x\in B$. 
\end{lemma}

This property is can be seen by equivariantly trivialising the restriction of $L$ to an affine chart of $B$.

\begin{example}
Fix a point $x \in X$, and consider two commuting one-parameter subgroups  $\lambda, \gamma$  of $G$. These induce a $(\C^*)^2$-equivariant morphism $\C^2 \to X$, where $\C^2$ is given the natural $(\C^*)^2$-action, and such that the image of $(\C^*)^2$ is contained in $G.x$. The image of $\C^2$ is the affine toric variety produced by taking the partial closure of the $(\C^*)^2$-orbit $(\C^*)^2.x \subset X$. The closure of the $(\C^*)^2$-orbit thus contains fixed points of $\gamma$ and $\lambda$, which we denote by $y_{\lambda, s}, y_{\gamma,t}$ respectively, and where we have parametrised using coordinates $s,t$ on $\C^2$. These families intersect at the image of the origin in $\C^2$:  $y_{\gamma, 0} = y_{\lambda, 0}$, as follows from commutativity of $\lambda, \gamma$. Equivariant constancy then implies that $\nu(y_{\lambda, s}, \lambda) = \nu(y_{\lambda, 0}, \lambda)$, and likewise for $\gamma$. This example is related to forthcoming work of Kirwan--Nanda on representations of $2$-quivers.
\end{example}

\subsection{Axiomatic stability conditions}\label{axiomatic} We wish to axiomatise the key structures of geometric invariant theory. Thus consider a projective scheme $X$ given the action of a reductive group $G$. Denote by $$\scS = \{ (y,\lambda) \ | \ y \in X \textrm{ and } \lambda \textrm{ is a one-parameter subgroup of } G_y\}/\sim,$$ where $(y,\lambda) \sim (z,\gamma)$ if there is a $g \in G$ with $g( y )= z$ and $g\cdot \lambda  = \gamma$. Thus $\scS$ is an upgrade of the space of orbits in $X$ by remembering also a one-parameter subgroup fixing the point in the orbit. There is a canonical injective map $$X \to \scS, \qquad x \to (x, \Id),$$ where  $\Id$ denotes the trivial one-parameter subgroup of $G$.

The set $\scS$ admits the structure of a directed graph, with vertices given by elements of $\scS$ and arrows given by equivariant specialisation. That is, we write $$(x, \zeta) \rightsquigarrow (y,\lambda)$$ if $\lambda$ commutes with $\zeta$ and $y$ is the specialisation of $x$ under $\lambda$. We write $x \rightsquigarrow  (y,\lambda)$ to mean that $(x, \Id) \rightsquigarrow (y,\lambda)$.

\begin{lemma} Equivariant specialisation is well-defined on orbits. \end{lemma}

\begin{proof} Let $(x,\zeta) \sim (z,\gamma)$, so that there is a $g \in G$ with $(g(x),g\cdot \zeta) =   (z,\gamma).$ If  $(x, \zeta) \rightsquigarrow (y,\lambda ),$ then    $$(z,\gamma) =  (g(x),g\cdot \zeta ) \rightsquigarrow (g(y),g \cdot \lambda) \sim (y, \lambda),$$ as required. \end{proof}

Fixing a maximal compact subgroup $K_x$ of $G_x$ with Lie algebra $\mfk_x$, as in Section \ref{classical-GIT}, it is then useful to consider $\scS$ as consisting of pairs $(x, v)$, where $v \in \mfk_{x}$ satisfies $\exp(v) = \Id$ and meaning that $v$ corresponds to a one-parameter subgroup of $G_x$.

\begin{definition}\label{definition: central charge stack equivariant main body} A \emph{central charge} is a function $Z: \scS \to \C$ such that:
\begin{enumerate}[(i)]
\item (Additivity) For a fixed $x \in X$, $Z$ induces a function $$Z(x,\cdot): \mfk_{x} \to \C$$ which is a Lie algebra character;
\item (Equivariant constancy) Suppose $B\subset X$ is a connected subscheme such that $\lambda$ is a one-parameter subgroup of $G_x$ for each $x \in X$. Then $Z(x,\lambda)$ is independent of $x \in X$.

\end{enumerate}
\end{definition}

Additivity in particular asks that $Z(x,\cdot )$ is additive under compositions of commuting one-parameter subgroups, so that it canonically extends to a linear function on $\mft_x = \Lie T_x$ for any maximal torus $T_x \subset K_x$. That $Z$ is well-defined on $\scS$ implies that $Z(x,v) = Z(x,g\cdot v)$ for $v \in \mfk_x$ and $g \in G_x$, so that its extension to $\mft_x$ induces an extension to $\mfk_x$ canonically, since maximal tori are always conjugate.

We record a motivating property for the definitions given here, for which we take $\zeta, \lambda$ commuting one-parameter subgroups. Their composition $\zeta \circ \lambda$ is then well-defined, and corresponds to exponentiating the sum of the associated elements of $\mfk$. In the following we view $Z$ as taking values on pairs of points and one-parameter subgroups.

\begin{lemma} Suppose $(x, \zeta) \rightsquigarrow (y,\lambda)$. Then $$Z(x,\zeta) + Z(y,\lambda) = Z(y, \zeta \circ \lambda).$$\end{lemma}

\begin{proof} Since $$Z(y, \zeta \circ \lambda) = Z(y,\zeta) + Z(y,\lambda),$$ the claim follows from the claim that $$Z(x,\zeta) = Z(y,\zeta),$$ which in turn follows from equivariant constancy.\end{proof}

This notion of a central charge allows us define stability, for which we must in addition fix a \emph{phase} $\phi \in (-\pi,\pi)$.

\begin{definition} We say that $x \in X $ is:
\begin{enumerate}[(i)]
\item \emph{$Z$-semistable} if for all $x \rightsquigarrow (y, \lambda)$ we have $$\Ima(e^{-i\phi}Z(y,\lambda)) \geq 0;$$
\item \emph{$Z$-polystable}  if for all $x \rightsquigarrow (y, \lambda)$ we have $\Ima(e^{-i\phi}Z(y,\lambda)) \geq 0$ with equality if and only if $G.x=G.y;$
\item \emph{$Z$-stable} if $x$ is polystable and in addition $G_x$ is finite;
\item \emph{$Z$-unstable} otherwise.
\end{enumerate}
\end{definition}

Thus stability is measured relative to the chosen phase $\phi$. $Z$-semistability and polystability can  be defined similarly for points $(x,\zeta) \in \scS$, and in this way $Z$-semistability of $(x,\zeta)$ asks for ``equivariant $Z$-semistability'' of $x$, namely $Z$-semistability with respect to one-parameter subgroups commuting with $\zeta$. In standard geometric invariant theory, equivariant semistability is equivalent to semistability, and so we choose not to emphasise this slightly more general situation.

\begin{remark}

Under an additional hypothesis on $\phi$ and $Z(y,\lambda)$, we can rephrase the numerical inequality governing stability in a way more reminiscent of Bridgeland stability \cite{bridgeland}. Let $\arg: \C \setminus \R_{\leq 0} \to (-\pi,\pi)$ denote the principal branch of the argument function, and denote by $$\Hbb = \{ z \in \C \ | \ \Ima(z) \geq 0 \textrm{ and } z \notin \R_{<0}\}.$$ Provided $Z(y,\lambda) \in \Hbb$ we define the \emph{phase} of $Z(x,\zeta)$ to be $$\phi(y,\lambda)= \arg((y,\lambda)) \in (-\pi,\pi).$$Suppose further that $e^{-i\phi}$ and  $Z(y,\lambda)$ both lie in $\Hbb$ (so that $\phi \in [0,\pi)$). Then the condition $$\Ima(e^{-i\phi}Z(y,\lambda))\geq 0$$ holds if and only if  the phase inequality $$\phi\geq \phi(y,\lambda)$$ holds, and similarly if one demands strict inequalities. In fact what is important is to fix a given half-plane in $\C$; the specific choice $\Hbb$ is not essential, as stability is independent of simultaneously rotating the phase $\phi$ and the central charge $Z$ by the same angle.
\end{remark}

\begin{remark} To obtain a working theory with standard properties---such as Zariski openness of the stable locus---one certainly needs further hypotheses on the central charge. What we have given here seems to be the minimum required to give a definition analogous to the usual definition of a central charge on an abelian category, as we discuss in Section \ref{sec:coherent}, and  to obtain a link with moment maps and differential geometry, as we discuss in Section \ref{sec:momentmaps}.
\end{remark}

\subsection{Stability conditions on stacks}\label{stacks} In Section \ref{axiomatic} we considered a scheme $X$ with the action of a reductive group $G$. Viewing this as producing a global quotient stack $[X/G]$, we next make analogous definitions for arbitrary stacks.  What will be important is that---just as with $[X/G]$---we can speak of a point along with a one-parameter subgroup of its stabiliser. Our discussion will be rather formal, but to match with the hypotheses of the appendix---where a more thorough stack-theoretic treatment is provided---we assume that $\X$ is a quasi-separated algebraic stack, which is locally of finite type over $\C$ and has affine stabilisers.

Thus let $\X$ be such a stack. For $y\in \X$  we denote by $\Aut(y)$ the automorphism group (or stabiliser) of $y$. We then define  $$\scS = \{ (y,\lambda) \ | \ y \in \X \textrm{ and } \lambda \textrm{ is a one-parameter subgroup of } \Aut(y)\}/\sim,$$ where we say $(x,\lambda) \sim (y, \gamma)$ if there is an isomorphism $g: x \cong y$ such that $g\cdot \lambda = \gamma$, where $g \cdot \lambda$ denotes the one-parameter subgroup of the stabiliser $G_y$ of $y$ under the isomorphism $g: x \cong y$.

We endow $\scS$ with the structure of a directed graph as follows. We begin with the simplest case, namely when we consider $(x,\Id)$, with $\Id$ denoting the trivial (identity) one-parameter subgroup. Here we say that $$(x, \Id ) \rightsquigarrow (y,\lambda)$$ if there is a morphism $\Psi: [\C/\C^*] \to \X$ such that $\Psi(1) = x$ and $$\Psi(0, \C^*) = (y, \lambda);$$ that is, we consider $\C^*\hookrightarrow \Aut(0)$ for $0 \in \C$ and ask that this one-parameter subgroup is mapped to $\lambda \hookrightarrow \Aut(y)$. 

 In general, to define what it means for there to exist a degeneration $$(x, \zeta) \rightsquigarrow (y,\lambda),$$ we consider  the quotient stack $[\C^2/(\C^*)^2]$ given the natural action. We choose coordinates $(s,t)$ on $\C^2$ and $\C^*$-actions $\hat \lambda, \hat\zeta$, so that for $s =1$ the point $(0,1)$ has stabiliser $\xi$ and similarly for $t =0$ the point $(1,0)$ has stabiliser $\hat\zeta$. The notation is suggestive: we will ask that $\hat \lambda, \hat\zeta$ map to $\lambda,  \zeta$ respectively.  We then say that $(x, \zeta ) \rightsquigarrow (y,\lambda)$ if there is a morphism $\Psi: [\C^2/(\C^*)^2] \to \X$ satisfying the following two conditions. Firstly, on points we ask that $\Psi(t,1) = x$ for all $t \in \C$ and $\Psi(t, 0) = y$. Secondly, on stabilisers we ask that the $\C^*$-stabiliser $\hat \zeta$ of $(0,1)$ is mapped to $\zeta \subset \Aut(\Psi(0,1)) = \Aut(x)$, and we in addition ask that the stabiliser $\hat \lambda$ of $(t,0)$ is mapped to $\lambda  \subset \Aut(\Psi(t,0)) = \Aut(y)$. The consideration of $\C^2$ is not essential, and one can instead consider the quotient stack $[\C/(\C^*)^2]$, where one copy of $\C^*$ lies in the stabiliser of every point of $\C$.

\begin{definition} We call a morphism $\Psi: [\C/\C^*] \to \X$ with $\Psi(1) = x$ a \emph{test configuration} for $x$.

\end{definition} 

\begin{remark} Morphisms $[\C/\C^*] \to \X$ are called ``filtrations'' of $x$ by Halpern-Leistner \cite{DHL}, motivated by  the case when $\X$ is the stack of coherent sheaves on a scheme, but since this word has various other meanings in the theory of K-stability, we instead use Donaldson's terminology.
\end{remark}
 
\begin{remark} 
With this terminology, maps $[\C^2/(\C^*)^2] \to \X$ (as considered in our definition of a degeneration $(x, \zeta) \rightsquigarrow (y,\lambda)$) are essentially ``test configurations of test configurations''; we give a geometric example of this condition for polarised schemes in Remark \ref{test-config-of-config}.
\end{remark}

We can now define a central charge analogously to before. Two pieces of notation will be useful. Firstly we let $ \mfk_{x} \subset \Lie \Aut(x)$ be the Lie subalgebra associated to a maximal compact subgroup $K_x \subset G_x$.

\begin{definition}\label{definition: central charge stack main body} A \emph{central charge} is a function $Z: \scS \to \C$ such that:
\begin{enumerate}[(i)]
\item (Additivity) For each $x \in \X$, $Z$ extends to a Lie algebra character $$Z(x,\cdot): \mfk_{x} \to \C.$$
\item (Equivariant constancy) Suppose $\Psi: [B/\C^*] \to  \X$ is a morphism where $\C^*$ fixes each point $b \in B$ of a connected finite-type scheme $B$, and denote by $\lambda_{\Psi(b)}$ the associated one-parameter subgroup of $\Aut(\Psi(b))$. Then $Z(\Psi(b),\lambda_{\Psi(b)})$ is independent of $b \in B$.\end{enumerate}
\end{definition}

\begin{remark}A central charge as defined here is a variant of  Halpern-Leistner's notion of a ``numerical invariant'' \cite{DHL}, and in particular the equivariant constancy used here is motivated by a property he demands.\end{remark}

A central charge induces a notion of stability just as before: fixing a phase $\phi \in (-\pi, \pi)$, we say that $x$ is \emph{$Z$-semistable} if for all $(x,\Id) \rightsquigarrow (y,\lambda)$ the inequality $$\Ima(e^{-i\phi}{Z(y,\lambda)}) \geq 0$$ holds, and \emph{$Z$-stability}, $Z$\emph{-polystability} and $Z$-\emph{instability} are defined analogously.

\subsection{Examples}\label{examples}
We give two examples of central charges for particular stacks: the stack of \emph{coherent sheaves} and the stack of \emph{polarised schemes}.

\subsubsection{Coherent sheaves}\label{sec:coherent}

Let $\scC$ denote the stack of coherent sheaves over a projective scheme $X$, with $\scS$ parametrising sheaves along $E$ along with a one-parameter subgroup of $\Aut(E)$.  There is a classical notion of a central charge on $\scC$: this associates to each coherent sheaf $E$ on $X$ a complex number $Z(E)$ which is deformation invariant (i.e. constant in flat families of sheaves; this is a consequence of central charges being assumed to factor through the numerical Grothendieck group of coherent sheaves and constancy of numerical invariants in flat families) and additive in short exact sequences. Here we explain how this canonically induces a central charge in the sense of Section \ref{stacks}, so that our notion can be seen as a generalisation of the classical notion.

In the stack of coherent sheaves, test configurations $[\C/\C^*] \to \scC$  for $E$ correspond to filtrations  of $E$  labelled by integers (this is standard; see e.g. \cite[Example 0.0.2]{DHL}). For example, any subsheaf $S\subset E$ induces the specialisation $$(E,\Id) \rightsquigarrow  (S\oplus E/S,(\Id, \exp(t)).$$

Suppose we are given for each coherent sheaf $E$ a complex number $Z(E)$ which is constant in flat families and additive in short exact sequences (such as from a central charge in the classical sense). For an element $$(E_1\oplus\hdots\oplus E_k, (\Id,\hdots, \exp(t),\hdots,\Id)) \in \scS,$$ with $\exp(t)$ in the $j\textsuperscript{th}$ spot, first set $$Z(E_1\oplus\hdots\oplus E_k, (\Id,\hdots, \exp(t),\hdots,\Id))  = Z(E_j).$$ Then  setting $$Z(E_1\oplus\hdots\oplus E_k, (\exp(a_1t),\hdots,\exp(a_kt)))  = \sum_{j=1}^k a_jZ(E_j)$$ induces a central charge: additivity follows by definition, while equivariant constancy is a consequence of $Z(E)$ being constant in flat families of sheaves.

The classical example of a central charge on  the stack of coherent sheaves over a polarised scheme $(X,L)$ is given by $$Z(E) = \rk E - i\deg E,$$ where $\rk E$ denotes the rank and $\deg E = c_1(E) \cdot L^{n-1}$ denotes the degree. With this choice $Z$-semistability recovers \emph{slope semistability}, while if one restricts to test configurations induced by \emph{saturated} subsheaves of $E$ then $Z$-polystability recovers slope polystability (for a survey explaining slope stability and its relation to Bridgeland stability see \cite{bayer-notes}).

\subsubsection{Polarised schemes} Denote by $\X$ the stack of $\Q$-polarised schemes with fixed Hilbert polynomial (that is, we consider schemes together with an ample $\Q$-line bundle). The main differences between the polarised scheme theory and the coherent sheaf theory is that test configurations no longer correspond directly to \emph{subobjects}, so it is more natural to consider polarised schemes with fixed Hilbert polynomial rather than considering all polarised schemes at the same time. 

Unravelling the definition, a test configuration for $(X,L)$ in $\X$ corresponds to a flat family $(\Y,\L_{\Y}) \to \C$ of polarised schemes along with a $\C^*$-action covering the natural one on $\C$, such that the fibres satisfy $(\Y_t,\L_{\Y,t}) \cong (X,L)$ for all $t \neq 0$. This agrees with the usual definition of a test configuration due to Donaldson \cite{donaldson-toric}, generalising Tian's prior work \cite{tian}. Given a test configuration with associated $\C^*$-action $\lambda$, we write $(X,L) \rightsquigarrow (\Y_0,\L_{\Y_0},\lambda)$.

The set $\scS$ consists of triples $(X,L,\zeta)$ where $(X,L)$ is a polarised scheme and $\zeta$ is a one-parameter subgroup of $\Aut(X,L)$. Denote by $\mfk_{(X,L)}$ the  Lie algebra of a maximal compact subgroup $K \subset \Aut(X,L)$. The notion of a central charge for polarised schemes is the following. 

\begin{definition}\label{def:equivariant-constancy-varieties} A central charge is a function $Z: \scS \to \C$ which satisfies:
\begin{enumerate}[(i)]
\item (Additivity) For a fixed polarised scheme, $Z$ induces a Lie algebra character $$Z((X,L),\cdot): \mfk_{(X,L)} \to \C.$$
\item (Equivariant constancy) Suppose $\pi: (\Y,\L_{\Y}) \to B$ is a flat family of polarised schemes, and suppose there is a $\C^*$-action $\lambda$ on $(\Y,\L_{\Y})$ such that $\pi \circ \lambda(t) = \pi$ for all $t$. Then $Z(\Y_b,\L_{\Y_b},\lambda)$ is independent of $b\in B$, where $(\Y_b,\L_{\Y_b})$ denotes the fibre of $\pi$ over $b$.\end{enumerate}
\end{definition}

Given a central charge, one can then ask for a polarised scheme $(X,L)$ to be \emph{$Z$-semistable, $Z$-stable, $Z$-polystable} or \emph{$Z$-unstable} in the natural way. For example, fixing a phase $\phi \in (-\pi, \pi)$ for $(X,L)$ to be $Z$-semistable asks that for all $(X,L) \rightsquigarrow (Y,L_Y,\lambda)$ we have $$\Ima\left(e^{-i\phi}Z(Y,L_Y, \lambda)\right) \geq 0.$$ For a test configuration $(\Y,\L)$, if we define $Z(\Y,\L_Y)$ to be $Z (\Y_0,\L_{\Y_0},\lambda)$, we may think of a central charge as associating a complex number to each test configuration, in an additive and equivariantly constant manner. With this notation $Z$-semistability then asks that for all test configurations $(\Y,\L)$ for $(X,L)$ we have $$\Ima\left(e^{-i\phi}{Z(\Y,\L)}\right) \geq 0.$$

\begin{remark}\label{test-config-of-config}

The analogue for polarised schemes of the condition that $(x,\zeta) \rightsquigarrow (y,\lambda)$ is the following. Suppose $(X,L)$ is a polarised scheme with $\zeta$ a $\C^*$-action on $(X,L)$. We can then ask for a test configuration $(\Y,\L_{\Y})$ to be $\zeta$-\emph{equivariant}, in the sense that there is a $\C^*$-action on $\Y$ acting fibrewise (so preserving the map to $\C$), extending the action of $\zeta$ on the general fibre $(X,L)$ and commuting with the $\xi$-action on $\Y$ coming from the definition of a test configuration. Given such a $\zeta$-equivariant test configuration, we obtain a family $
\Y \times \C \to \C^2$ with a $(\C^*)^2$-action induced from $\xi$ on the first factor and (the extension of) $\zeta$ on the second. There are two $\C^*$-actions on $(\Y_0,\L_{\Y_0})$ induced by $\xi$ and $\zeta$, and the condition that $(X,L,\zeta) \rightsquigarrow (\Y_0,\L_{\Y_0}, \lambda)$ asks that there is a $\zeta$-equivariant test configuration as described such that in addition $\xi \circ \zeta = \zeta \circ \xi =\lambda$ as subgroups of $\Aut (\Y_0,\L_{\Y_0}).$ Note that this implies $\lambda$ and $\zeta$ also commute in $\Aut (\Y_0,\L_{\Y_0}).$ As before, it is not essential to consider families over $\C^2$, as one can instead consider families over $\C$ at the expense of working with the ineffective quotient $[\C/(\C^*)^2]$.

\end{remark}

\begin{example}[K-stability] Suppose $(\Y,\L)$ is a test configuration for an $n$-dimensional scheme $(Y,L)$, inducing a $\C^*$-action $\lambda_0$ on $H^0(\Y_0,\L_0^k)$ for all $k$. The dimension of $H^0(\Y_0,\L_0^k)$ and the total weight of the action on $H^0(\Y_0,\L_0^k)$ are polynomials for $k \gg 0$ which we may write \begin{align*} h(k) &= a_0k^n + a_1k^{n-1} + O(k^{n-2}), \\ w(k) &= b_0k^{n+1}+b_1k^n+O(k^{n-1})\end{align*} respectively. Setting $$Z((\Y_0,\L_0),\lambda_0) = -ib_0 + b_1, \qquad Z((Y,L),\Id) = ia_0-a_1,$$ produces a central charge, and the notion of $Z$-semistability recovers Donaldson's notion of K-semistability \cite{donaldson-toric} (extending Tian's analytic definition in the Fano case \cite{tian}), and similarly for $Z$-stability and $Z$-polystability. 

After Donaldson's original work,  a subtlety in the definition of K-stability (rather than K-semistability) was realised: for normal varieties one must exclude certain ``almost trivial'' test configurations (test configurations whose total space normalises to the trivial test configuration) to have a sensible theory \cite{li-xu, stoppa, uniform, BHJ}; it is not clear what role these degenerate test configurations play in the theory for more general central charges.  Excluding almost trivial test configurations is analogous to the restriction to considering saturated subsheaves in the definition of slope stability of torsion-free coherent sheaves; by contrast torsion sheaves play a central role in Bridgeland stability.

\end{example}

\begin{example} The notion of a central charge is intended to axiomatise the notion of a central charge introduced in \cite{stabilityconditions}. There a specific smooth polarised variety $(X,L)$ was fixed, and  a central charge was defined explicitly through a choice of topological information on $(X,L)$. This choice canonically induces a phase $\phi = \arg Z(X,L)$, which is then independent of $(X,L)$ itself provided it varies in a flat manner. For each test configuration $(\Y,\L)$ with smooth total space, a number $Z(\Y,\L) \in \C$ was then defined via intersection theory on a natural compactification of the total space $\Y$. Thus the definition relies on $(\Y_0,\L_0)$ being the central fibre of a test configuration with reasonable total space. It would be  interesting to define these quantities intrinsically on $(\Y_0,\L_0)$, more in line with the perspective of this note. \end{example}

\section{$Z$-critical points and complex moment maps}\label{sec:momentmaps}

We next describe the analytic counterpart to $Z$-stability, through what we call \emph{complex} moment maps. In the traditional theory of moment maps,  one can either consider maps to the  Lie algebra  or its dual; the former requires a choice of  an inner product. This is mostly an aesthetic choice, and we choose to  fix an inner product so that the  links with the motivating \emph{infinite}-dimensional problems are most transparent, as this is one of the main goals of our work. We refer to Kirwan \cite{kirwan}, Georgoulas--Robbin--Salamon \cite{GRS} and Hashimoto \cite{hashimoto} for comprehensive accounts of the relationship between moment maps and geometric invariant theory.

To define the inner product, we proceed as follows. Firstly we fix a faithful representation of $G$ on a complex vector space $V$, giving an embedding $$\mfg \subset \End V.$$ Thus we may \emph{multiply} elements of $\mfg$. The most important example to keep in mind is when $X \subset \pr^n = \pr(V)$ is a subvariety of projective space and $G$ acts faithfully and linearly on projective space, meaning there is a natural $G$-action on $V$. We assume that $G$ is reductive, and is hence the complexification of a maximal compact subgroup $K \subset G$. We also fix a $K$-invariant Hermitian inner product on $V$, which induces one on $\End V$ and hence induces an isomorphism $$\mfg \cong \mfg^*.$$ In this way, for $u \in \mfg$ and $\alpha \in \mfg^*$ we have $$\langle u, \alpha \rangle = \tr (u^*\alpha^{\lor}),$$ where $\langle \cdot , \cdot \rangle$ is the natural pairing, $\alpha^{\lor} \in \mfg$ is the dual element of $\alpha$ and $u^*$ is the congugate transpose. 

We now return to a smooth projective variety $X$, which we endow with a closed $K$-invariant \emph{complex} $(1,1)$-form $\omega_Z$. For the moment we do \emph{not} assume any positivity hypotheses on $\omega_Z$. 

Let $G$ be a reductive linear algebraic group acting holomorphically on $X$ and fix a representation of $G$ and a Hermitian inner product as above. We write $K$ for the maximal compact subgroup of $G$, and let $\mfk$ denote the Lie algebra of $K$. We will---slightly abusively---identify an element $v \in \mfg$ with its induced vector field on $X$.

\begin{definition} We say that a smooth map $$\tilde Z: X \to \mfg$$ is a \emph{complex moment map} if for all $u \in \mfk$ we have $$d \tr \left(u^* \tilde Z\right) = -\iota_{u} \omega_Z$$ and $\tilde Z$ is $K$-equivariant with respect to the adjoint action on $\mfk$. 
 \end{definition}

 \begin{remark} 
 
 Similar structures to complex moment maps arise in B\'erczi--Kirwan's recent work providing a  moment map interpretation of non-reductive GIT \cite{berczi-kirwan}, and it would be interesting to understand the relationship between their work and what we consider here.
 
 \end{remark}
 
We now assume that $\tilde Z$ is a complex moment map. To link with the definition of a central charge, it is useful to upgrade $\tilde Z$ to a smooth function $$\tilde Z: X \times \mfk \to \mfg$$ by defining $$\tilde Z(x, v) =v^*\tilde Z(x),$$ where the second term is interpreted as multiplication of elements of $\mfg \subset \End V$. In the following we view $\scS$ as consisting of pairs $(x,u)$ such that $x \in X$ and $u \in \mfk_{x}$. We fix a central charge $Z$ on $\S$ and a phase $\phi \in (-\pi,\pi)$.

\begin{definition}\label{def:compatibility}
We say that $\tilde Z$ is \emph{compatible with a central charge} $Z: \scS \to \C$ if for all $(x,v) \in \scS$ we have $$\tr(\tilde Z(x,v)) = Z(x,v)$$ and $$\arg \tr(\tilde Z(x)) = \phi.$$
\end{definition}

\begin{remark} Compatibility is a key point of the definitions: if one thinks of $Z$ as analogous to a choice of topological classes---as will be the case in the examples given below---compatibility is analogous to asking that $\tilde Z$ produces Chern-Weil and equivariant Chern-Weil representatives of these topological classes. Being able to phrase the compatibility condition is the main advantage of choosing a representation of $G$ on the vector space $V$.
\end{remark}

From here we fix a central charge $Z$ compatible with the complex moment map $\tilde Z$. We thus turn to linking complex moment maps with $Z$-stability, and in particular we will exclusively be interested in understanding $\tilde Z$ on points, rather than general pairs $(x,u)$. We fix a phase $\phi \in (-\pi,\pi)$.

\begin{definition} We say that a point $x$ is \emph{$Z$-critical} if $$\Ima(e^{-i\phi} \tilde Z(x)) = 0,$$ where $\Ima$ refers to the skew-Hermitian part of an element of $\mfg$ with respect to the Hermitian inner product. \end{definition}

As we will explain, this is the key condition related to $Z$-polystability. A  basic observation shows compatibility is actually automatic at $Z$-critical points:

\begin{lemma}\label{compatibility-lemma}

Suppose $x$ is a $Z$-critical point. Then $$\arg \tr(\tilde Z(x)) = \phi.$$

\end{lemma}

\begin{proof}

Write $$\tilde Z(x) = M_h(x) + M_s(x)$$ where $e^{-i\phi}M_h(x)$ is Hermitian and $e^{-i\phi}M_s(x)$ is skew-Hermitian. Then we see that $$ \Ima( e^{-i\phi}  \tilde Z(x)) = e^{-i\phi} M_s(x),$$ so  \begin{align*}\tr \Ima( e^{-i\phi}  \tilde Z(x)) &= \tr e^{-i\phi} M_s(x), \\  &= e^{-i\phi} \tr M_s(x), \\ &= \Ima(  e^{-i\phi} \tr \tilde Z(x)),\end{align*} where in the final step we used the similar fact that $\tr e^{-i\phi} M_h(x)$ is real.  Thus if $\Ima(e^{-i\phi} \tilde Z(x)) = 0,$ we must have $\Ima(  e^{-i\phi} \tr \tilde Z(x)) = 0,$ which  implies $$\arg \tr(\tilde Z(x)) = \phi,$$ concluding the proof.
\end{proof}

While we  actually assume compatibility throughout, this result nevertheless makes clear that it is a natural condition. We next relate complex moment maps to usual moment maps for compact group actions.

\begin{definition} We call a map $\mu: X \to \mfk^*$ a \emph{formal moment map} with respect to a closed $(1,1)$-form $\eta$ if $\mu$ is K-equivariant and $$d\langle \mu, v\rangle = -\iota_v \eta.$$\end{definition}

 Thus if $\eta$ is positive---hence defining a symplectic form---$\mu$ is a  moment map in the usual sense. In the language of equivariant cohomology, the condition asks that the (complex) equivariant differential form $\eta + \mu$  is equivariantly closed.

\begin{proposition} Suppose $\tilde Z$ is a complex moment map with respect to $\omega_Z$. Then $$i\Ima(e^{-i\phi}\tilde Z(\cdot))^{\lor}: X \to \mfk^*$$ is a formal moment map with respect to the $(1,1)$-form $\Rea(e^{-i\phi}\omega_Z).$

\end{proposition}

Here the notation $i\Ima(e^{-i\phi}\tilde Z(\cdot))^{\lor}$ means the composition of $i\Ima(e^{-i\phi}\tilde Z(\cdot)) \to \mfk$ with the isomorphism $\mfk \cong \mfk^*$, while $\Ima(e^{-i\phi}\omega_Z)$ denotes the imaginary part of the complex $(1,1)$-form $e^{-i\phi}\omega_Z$. Note that $\Ima(e^{-i\phi}\tilde Z(\cdot))^{\lor}$ itself has image in $i\mfk$, meaning $i\Ima(e^{-i\phi}\tilde Z(\cdot))^{\lor}$ takes values in Hermitian matrices. This extra factor of $i$ is compensated for in the $(1,1)$-form as $$\Rea(e^{-i\phi}\omega_Z) = \Ima (e^{-i\phi}i\omega_Z) .$$

\begin{proof}

$K$-equivariance of $\tilde Z: X \to \mfg$ implies that $i\Ima(e^{-i\phi}\tilde Z(\cdot)): X \to \mfk$ is $K$-equivariant, while the $K$-equivariance of the Hermitian inner product on $V$ implies that the isomorphism $\mfk \cong \mfk^*$ is $K$-equivariant. 

To prove the moment map equation, it is equivalent to show that  $$d \langle u, \Ima(e^{-i\phi}\tilde Z(x))^{\lor}\rangle = -\iota_{u}  \Ima(e^{-i\phi}\omega_Z).$$ Since $u \in \mfk$ corresponds to a real vector field on $X$, we have $$ \iota_{u}  \Ima(e^{-i\phi}\omega_Z) = \Ima(e^{-i\phi}\iota_{u} \omega_Z),$$ which by the complex moment map identity gives $$ -\iota_{u}  \Ima(e^{-i\phi}\omega_Z) =  \Ima\left(e^{-i\phi}d \tr\left(u^* \tilde Z\right)\right).$$ Using a similar linear algebra argument as Lemma \ref{compatibility-lemma} along with the fact that $u$ viewed as an element of $\End V$ corresponds to a Hermitian matrix, we see  $$ \Ima\left(e^{-i\phi}d \tr\left(u^* \tilde Z\right)\right) =  d \tr \left(u^* \Ima(e^{-i\phi}\tilde Z(x))\right).$$Since the isomorphism $\mfk \cong \mfk^*$ arises from the Hermitian inner product on $V$, it follows that $$\tr \left(u^*\Ima(e^{-i\phi}\tilde Z(x)) \right) = \langle u, \Ima( e^{-i\phi}\tilde Z(x))^{\lor}\rangle$$ and hence $$d \langle u, \Ima(e^{-i\phi}\tilde Z(x))^{\lor}\rangle = -\iota_{u} \Ima(e^{-i\phi}\omega_Z),$$ proving the result. \end{proof}

Positivity is, of course, crucial to the theory of moment maps. While many aspects of the theory require \emph{global} positivity, others rely only on \emph{local} positivity; for example, to obtain a symplectic quotient, one only needs positivity in an open neighbourhood of the zero-set of the moment map.

\begin{definition} We say that a point $x\in X$ is a $Z$-\emph{subsolution} if the form $$\Rea(e^{-i\phi}\omega_Z)$$ is positive on $T_xX$, in the sense that $$\Rea(e^{-i\phi}\omega_Z)(u, J_xu)>0$$ for all $u\neq 0$ with $J_x: T_xX \to T_xX$ being the almost complex structure. We further say that $\tilde Z$ satisfies the \emph{global subsolution hypothesis} if every point $x\in X$ is a subsolution.
\end{definition}

The global subsolution hypothesis is strong: analogues fail in infinite-dimensions, as discussed in the Remark \ref{rmk:subsolutions}. As mentioned there, in the better-understood infinite-dimensional problems, what is expected to be true is that every solution of the equation (i.e. being an analogue of a $Z$-critical point) \emph{is} also a subsolution. This is often enough to obtain geometric consequences:

\begin{theorem}\label{thm:quotient} Suppose every $Z$-critical point is a $Z$-subsolution. Then the symplectic quotient $$X/_Z K := \Ima(e^{-i\phi}\tilde Z(\cdot ))^{-1}(0)/K$$ admits the structure of a K\"ahler space. 
\end{theorem}

\begin{proof} This is classical under the global subsolution hypothesis \cite{HL,HHL}, but the proofs only require a K\"ahler metric in a \emph{neighbourhood} of the zero set of the moment map. Thus since every $Z$-critical point is a subsolution, and the subsolution condition is \emph{open} in $x$ (as it is a positivity condition on an inner product on the tangent space), the form $\Rea(e^{-i\phi}\omega_Z)$ is indeed a K\"ahler metric in a neighbourhood of $\Ima(e^{-i\phi}\tilde Z(\cdot ))^{-1}(0).$
\end{proof}

More explicitly, in this generality the K\"ahler metric on the quotient is produced as follows. Let $\psi_Z \in C^{\infty}(X,\C)$ be a local potential for the complex $(1,1)$-form $\omega_Z$ in a neighbourhood of a $Z$-critical point $x$, in the sense that near $x$ $$\omega_Z = i\ddbar \psi_Z.$$ Restricting $\psi_Z$ to (a neighbourhood of $x$ intersected with) $\Ima(e^{-i\phi}\tilde Z(\cdot ))^{-1}(0)$, as $\psi_Z$ is a $K$-invariant function, it descends to a continuous function $\tilde \psi_Z$ on the quotient $\Ima(e^{-i\phi}\tilde Z(\cdot ))^{-1}(0)/K$. The function $$\Rea(e^{-i\phi} \tilde \psi_Z) \in C^0(X/_Z K, \R)$$ is then a weak K\"ahler potential for the induced form on the complex space $X/_Z K$ in the sense used by Heinzner--Huckleberry--Loose \cite{HHL}.

Our main result explains how to relate the existence of $Z$-critical points to complex moment maps. The proof reduces to a version of the classical Kempf--Ness theorem (due Kempf--Ness in the affine setting \cite{kempf-ness} and to Kirwan in the projective setting \cite{kirwan}).

\begin{theorem} Suppose $\tilde Z$ satisfies the global subsolution hypothesis. Then the following are equivalent:
\begin{enumerate}[(i)]
\item there is a point $y \in G.x$ such that $\Ima(e^{-i\phi}\tilde Z(y)) = 0$;
\item $x$ is $Z$-polystable.
\end{enumerate}

\end{theorem}

\begin{proof}

By the global subsolution hypothesis, the form $\Ima(e^{-i\phi}\omega_Z)$ is a K\"ahler metric on $X$ and $\Ima(e^{-i\phi}\tilde Z(\cdot ))$ is a moment map. Although $X$ is a smooth projective variety, the form $\Rea(e^{-i\phi}\omega_Z)$ may not lie in an integral class, meaning we cannot apply the classical Kempf--Ness theorem. Instead we apply the Kempf--Ness theorem for \emph{K\"ahler} manifolds (see for example the survey \cite[Section 12]{GRS}), which implies that the existence of a $Z$-critical point in the orbit of $x$ is equivalent to the condition that for all $x \rightsquigarrow (y,u)$ we have $$\langle i\Ima(e^{-i\phi}\tilde Z(y))^{\lor}, u \rangle \leq 0,$$  with equality if and only if $y=x$. Here our slightly extended notation means that $$y = \lim_{t\to \infty} \exp(-itu).x.$$ What this essentially means is that, in the K\"ahler setting one must also include ``irrational'' vector fields to obtain the existence of a zero of the moment map, rather than merely rational ones inducing one-parameter subgroups of $G$. 

The rest of the proof will compare this numerical condition to the one governing $Z$-polystability, and will then explain that in fact it \emph{is} enough to merely consider ``rational'' vector fields (equivalently one-parameter subgroups) in our situation.

The definition of the isomorphism $\mfg \cong \mfg^*$ gives \begin{align*}\langle i\Ima (e^{-i\phi}\tilde Z(y)^{\lor}), u \rangle  &= \tr ((i\Ima (e^{-i\phi}\tilde Z(y)))^*u ), \\ &= \tr (i\Ima(e^{-i\phi}\tilde Z(y) u), \\ &= -\Ima(e^{-i\phi}Z(y, u)). \end{align*} In slightly more detail, the final step follows from the fact that for a (complex) matrix $A$ and a Hermitian matrix $B$ we have $$\tr (i(\Ima A) B) =- \Ima \tr (AB),$$ and also from the compatibility condition $\tr \left(u^*\tilde Z (y)\right) = \tr \left(\tilde Z (y) u \right) =Z(y,u)$ again using that $u$ is Hermitian.

To conclude we must show that it is enough to check stability with respect to \emph{rational} elements of $\mfk$, generating one-parameter subgroups. If $x$ is $Z$-unstable, it is standard to produce rational destabilising elements of $\mfk$ given the existence of an irrational one by an approximation argument, thus we may assume that $x$ is $Z$-semistable; here we note that $Z$-semistability with respect to rational vector fields implies $Z$-semistability also with respect to irrational ones. Then from the ``semistable'' case of the Kempf-Ness theorem there is a point $z \in \overline{G.x}$ which is itself $Z$-critical. We can then take a slice of the $G$-action in a neighbourhood of the $Z$-critical point $z$ (essentially by construction of the quotient in the complex setting \cite{HL}), so that the action is modelled on the linear action on $T_z X$, where it is clear that one can find a one-parameter subgroup taking $x$ to $z$. \end{proof}

\begin{remark}
Rather than the global subsolution hypothesis, the proof only requires the weaker condition that $\Rea(e^{-i\phi}\omega_Z)$ be positive in a neighbourhood of $\overline{G.x}$.\end{remark}

In the classical projective case a consequence of this sort of result is a homeomorphism between the symplectic and algebraic quotients (this is due to Kirwan \cite{kirwan}). As we have appealed to a K\"ahler version of the Kempf--Ness theorem, there is no purely algebraic definition of the quotient. So while---under the global subsolution hypothesis---one still obtains a quotient $X/_K K$ which is a complex space endowed with a K\"ahler metric by Theorem \ref{thm:quotient}, there is no direct algebraic construction to compare it with.

\subsection{Examples in infinite dimensions}
We next briefly explain the link between the categorical notions of stability for coherent sheaves and polarised schemes and moment maps.

\subsubsection{Vector bundles} Associated to a class of central charges on $\Coh(X)$---which in particular take the form of Section \ref{sec:coherent}---is a partial differential equation on Hermitian metrics on holomorphic vector bundles on $X$, solutions of which are called \emph{$Z$-critical connections} \cite{DMS}. Briefly, these central charges involve a choice of K\"ahler class on $X$, a choice of (products of) Chern classes of the sheaf and a choice of topological classes on $X$ (as motivated by Bayer's polynomial stability conditions \cite{bayer}). To a K\"ahler metric on $X$, a Hermitian metric on $E$ producing a Chern connection A and a closed differential form on $X$ representing the topological class is then associated an $\End E$-valued $(n,n)$ form $\tilde Z(E,A)$ which satisfies $$\int_X \tilde Z(A) = Z(E),$$ which is analogous to the compatibility of the central charge as given in Definition \ref{def:compatibility}. The inner product corresponding to the trace used there is  the $L^2$-inner product defined with respect to the volume form associated to the K\"ahler metric.

The $Z$-critical equation then asks that $$\Ima(e^{-i\phi(E)} \tilde Z(h)) = 0,$$ where this denotes the skew-Hermitian part of the $\End E$ component as defined through the Hermitian metric $h$. The various sign conventions used in the present work are chosen to match with the $Z$-critical connection and deformed Hermitian Yang--Mills literature. An especially noteworthy example is given by the deformed Hermitian Yang-Mills equation \cite{LYZ, MMMS} (which appeared long before the more general notion of a $Z$-critical connection), which corresponds to the central charge $$Z(E) = \int_X e^{-i[\omega]}\cdot \ch(E);$$ other equations relevant to string theory and mirror symmetry involve including Chern classes of  $X$ itself \cite[Example 2.8]{DMS}.

\begin{remark}\label{rmk:subsolutions} On line bundles, an almost complete theory of deformed Hermitian Yang--Mills connections exists, especially due to Chen, Collins, Jacob and Yau \cite{chen, jacob-yau, collins-jacob-yau, collins-yau}, and their theory emphasises many of the structures one should expect to be relatively general. For example, the existence of a solution to the deformed Hermitian Yang-Mills equation \emph{implies} that the associated Hermitian metric is a subsolution \cite{collins-jacob-yau} (in the so-called ``supercritical phase range'').  An important aspect of the theory this same statement (``solution implies subsolution'') is true along a continuity method one can use to solve the equation under a stability hypothesis \cite{chen}, so one always has positivity along the path designed to solve the equation.

In fact, the \emph{converse} of this statement also holds on line bundles: the existence of deformed Hermitian Yang--Mills connections on a line bundle is equivalent to the existence of a subsolution (again in the appropriate phase range); this is due to Collins--Jacob--Yau \cite{collins-jacob-yau}. In higher rank, for appropriate classes of central charge it should still be the case that the existence of a solution implies the existence of a subsolution, but the converse cannot hold. For example, the Hermitian Yang-Mills condition is a special case of the $Z$-critical condition, and here the subsolution condition is automatic (as it asks that $\omega^{n-1}$ is a positive $(n-1,n-1)$-form where $\omega$ is the K\"ahler metric), but nevertheless obstructions to the existence of solutions appear from saturated subsheaves.
\end{remark}

\subsubsection{Polarised varieties} The theory for smooth polarised varieties is analogous, but with additional complications on the analytic side \cite{stabilityconditions,momentmaps}. Here the equation is for a K\"ahler metric $\omega \in c_1(L)$ on a smooth projective variety $X$, and one makes analogous choices---namely a choice of topological classes on $X$ and products of Chern characters of $X$. The equation is only explicitly available in the case of powers of the \emph{first} Chern class of $X$ (along with arbitrary auxiliary differential forms on $X$ and powers of the ample line bundle), where one associates to $\omega$ a complex valued function $$\tilde Z(\omega): X \to \C.$$ To tighten the parallel with the bundle theory, equivalently by multiplying by $\omega^n$ one can consider $\tilde Z(\omega)$ as a complex valued $(n,n)$-form. This complex $(n,n)$-form satisfies the ``compatibility'' condition $$\int_X \tilde Z(\omega) = Z(X,L),$$ and the $Z$-critical equation asks $$\Ima(e^{-i\phi(X,L)}\tilde Z(\omega)) = 0.$$

The actual construction of $\tilde Z(\omega)$, however, is more subtle than its bundle analogue. The reason is that its construction involves not only various Chern--Weil representatives, but also higher-order terms essential for a link with algebraic geometry. A good understanding of the $Z$-subsolution condition---along with various other foundational structures---remains to be achieved.

\subsection{Structures in complex moment map theory}\label{structures}

We now briefly explain the appearance of several standard structures in classical moment map theory in our  setup: the norm-squared of the moment map, the moment map flow, the log-norm functional and the log-norm functional as a K\"ahler potential. Many of these have appeared in the infinite-dimensional theories discussed above, and our new perspective gives some finite-dimensional motivation for their appearance. All of these structures are discussed at great length in the survey \cite{GRS}.

For any $x\in X$, the norm-squared of the moment map is simply the value $$\|\Ima(e^{-i\phi}\tilde Z(x))\|^2 =  \tr (\Ima(e^{-i\phi}\tilde Z(x)) \cdot \Ima(e^{-i\phi}\tilde Z(x))).$$ This is the functional whose Euler--Lagrange equation produces both $Z$-critical points and $Z$-\emph{extremal points}: points which satisfy $$ \Ima(e^{-i\phi}\tilde Z(x)) \in i\mfk_x.$$

To define the moment map flow, note that for $x \in X$ the value  $i\Ima(e^{-i\phi}\tilde Z(x) \in \mfk$ can be thought of as an element of $T_xX$ through the infinitesimal action. Thus for any $x_0$ we may define a flow by $x(0) = x_0$ and $$\frac{d}{dt}x(t) = -\Ima(e^{-i\phi}\tilde Z(x(t))).$$ This is the downward gradient flow of the norm-squared of the moment map $$x \to \|\Ima(e^{-i\phi}\tilde Z(x))\|^2,$$ and we call this flow the $Z$-\emph{flow}. The asymptotics of this flow are related to ``optimal destabilising one-parameter subgroups'', which are in turn analogous to Harder--Narasihman-type filtrations in the coherent sheaf setting. In the deformed Hermitian Yang--Mills setting this flow corresponds to the \emph{tangent Lagrangian phase flow} of Takahashi \cite{takahashi-flow}, we note that in that setting there is also the \emph{line bundle mean curvature flow}, introduced by Jacob--Yau \cite{jacob-yau}, which is instead motivated by the \emph{Lagrangian mean curvature flow} in the study of special Lagrangians.

The log-norm functional is a functional on a fixed orbit, which is defined through its variation. Fixing a reference point $x\in X$, any other point in the form $g.x$ for $g \in G$. We first define a one-form $d E_Z$ on $G$ (the notation will be justified by this one-form being exact) by setting $$\langle dE_Z, u\rangle_g = \tr  (u^* \tilde Z(g(t).x)).$$  This is then $K$-invariant and hence descends to a one-form on the symmetric space $G/K$. A standard calculation, identical to the usual one in moment map theory, shows that this one form is closed and is hence exact. In particular it is well-defined, independent of choice of path. Thus we obtain a functional $$E_Z: G/K \to \C,$$ and we define the \emph{$Z$-energy} to be $$\Rea(e^{-i\phi} E_Z): G/K \to \R.$$ This is the analogue of the log-norm functional; this is convex along geodesics in the symmetric space $G/K$ in the locus of $Z$-subsolutions, and it is strictly decreasing along the $Z$-flow (again in the locus of $Z$-subsolutions). In the deformed Hermitian Yang--Mills setting this corresponds to what Collins--Yau call the \emph{Calabi--Yau functional}  \cite[Definition 2.13]{collins-yau} and in the setting of $Z$-critical K\"ahler metrics corresponds to the $Z$\emph{-energy} \cite[Definition 3.7]{stabilityconditions}.

We lastly turn to the potential for the form $\omega_Z$. A $G$-orbit $G.x \subset X$ is affine, hence on this locus $\omega_Z = i\ddbar \psi_Z$ for some complex-valued function $\psi_Z$. We can consider $E_Z$ as a function $G.x \to \R$ by defining the $Z$-energy relative to the base-point $x$. Then on this locus a calculation shows that $$i\ddbar E_Z = \omega_Z,$$ so that we can view the $Z$-energy as a potential for the form $\omega_Z$. In particular on this locus we have $$i\ddbar (\Rea(e^{-i\phi} E_Z)) = \Rea(e^{-i\phi}\omega_Z).$$ Thus the $Z$-subsolution condition forces the complex Hessian $i\ddbar (\Rea(e^{-i\phi} E_Z))$ to be positive at the point $x$.

\section*{Appendix, by Andr\'es Ib\'a\~nez N\'u\~nez}

In this appendix we explain how the notion of a central charge on an algebraic stack $\cX$ can be formulated using the formalism of graded points on $\cX$, in the spirit of Halpern-Leistner's definition of numerical invariant \cite{DHL}. In the present section we will call these \emph{complex linear forms} on $\X$, to make the statements of our results (especially the equivalence with the Definition \ref{definition: central charge stack equivariant main body}) transparent.

We denote $\Gbb_m=\Spec \C[t,t^{-1}]$ the multiplicative group scheme over $\C$, whose $\C$-points are $\Gbb_m(\C)=\C^*$. A crucial role will be played by the classifying stack $B\Gbb_m$ of $\Gbb_m$. While $B\Gbb_m$ is determined by the fact that, for any algebraic stack $\cX$ over $\C$, the groupoid $\Hom(\cX,B\Gbb_m)$ of maps $\cX \to B\Gbb_m$ is equivalent to that of line bundles on $\cX$, here we will rather be interested in maps \emph{from} $B\Gbb_m$ into other algebraic stacks.

We fix an algebraic stack $\cX$, quasi-separated and locally of finite type over $\C$, with affine stabilisers. Examples of stacks satisfying these assumptions are moduli stacks of polarised projective schemes over $\C$ \cite[\href{https://stacks.math.columbia.edu/tag/0DPS}{Tag 0DPS}]{stacks-project}, \cite[Section 2.1]{starr-dejong}, moduli stacks of objects in suitable $\C$-abelian categories \cite[Section 7]{existencemoduli} and stacks of $G$-bundles on a proper scheme $X$ over $\C$ for $G$ a linear algebraic group over $\C$ \cite[\href{https://book.themoduli.space/tag/00BK}{Tag 00BK}]{modulibook}.

\begin{definition}\label{definition: stack of graded points}
The \emph{stack of graded points} $\Grad(\cX)$ of $\cX$ is the stack over $\C$ defined by setting, for any scheme $T$ over $\C$,
\[\Hom(T,\Grad(\cX))=\Hom(B\Gbb_m\times T,\cX).\]
\end{definition}

In other words, $\Grad(\cX)$ is the mapping stack $\uMaps(B\Gbb_m,\cX)$. Thus a map $T\to \Grad(\cX)$ is the same data as a map $B\Gbb_m\times T\to \cX$. It is a nontrivial result \cite[Theorem 5.10]{alper} that $\Grad(\cX)$ is an algebraic stack locally of finite type over $\C$. We denote $\abs{\Grad(\cX)}$ its underlying topological space. 

We will use the notation $\Gamma^\Z(\Gbb_{m}^n)= \Hom(\Gbb_{m},\Gbb_{m}^n)$ for the group of cocharacters of $\Gbb_{m}^n$, which is isomorphic to $\Z^n$. More generally, for a linear algebraic group $G$ over $\C$ we denote $\Gamma^\Z(G)$ the set of cocharacters of $G$ and $\Gamma_\Z(G)$ the abelian group of characters of $G$.

 The main definition in this appendix is:

\begin{definition}\label{definition: central charge on a stack}
A \emph{complex linear form} $Z$ on $\cX$ is a locally constant map
\[Z\colon \abs{\Grad(\cX)}\to \C\]
such that, for any $n\in \Z_{>0}$ and any morphism $g\colon B\Gbb_{m}^n\to \cX$, the map $$Z_g\colon \Gamma^\Z(\Gbb_{m}^n)\to \C$$ induced by $g$ and $Z$ is $\Z$-linear, where the definition of $Z_g$ is as follows: if $\alpha\colon \Gbb_{m}\to \Gbb_{m}^n$ is a cocharacter, the composition 

\[\begin{tikzcd}
 B\Gbb_{m} \arrow[r, "B\alpha"]& B\Gbb_{m}^n \arrow[r, "g"]& \cX
\end{tikzcd}\]
defines a point $p$ of $\Grad(\cX)$, and we set $Z_g(\alpha)=Z(p)$.

We denote $\Charges(\cX)$ the set of complex linear forms on $\cX$, which is naturally a $\C$-vector space. The notation is intended to signify that we think of a complex linear form as a ``pre-stability condition'', where the eventual full structure of a stability condition should in addition require a positivity property.
The relationship with central charges and stability in the sense of Sections \ref{stacks} and \ref{axiomatic} is explained by Remarks \ref{equivalence of definitions of central charge} and \ref{connection}, respectively.

\end{definition}

\begin{remark}
Definition \ref{definition: central charge on a stack} makes sense whenever $\Grad(\cX)$ is an algebraic stack, which holds under very general assumptions on $\cX$ (see \cite[Theorem 5.10]{alper}). In this generality, the linearity condition in the definition should be imposed for all $g\colon B\Gbb_{m,k}^n\to \cX$, where $k$ is an arbitrary algebraically closed field.
\end{remark}

Complex linear forms on the classifying stack of a group have a transparent description.

\begin{lemma}\label{lemma: central charges for BG}
Let $G$ be an affine algebraic group over $\C$. Then there is a canonical isomorphism
\[\Charges(BG)\cong \C\otimes_\Z \Gamma_\Z(G)\]
between the vector space of complex linear forms on $BG$ and that of complex characters of $G$.

\end{lemma}
\begin{proof}
Let $T$ be a maximal torus of $G$ and let $W=N_G(T)/C_G(T)$ be the associated Weyl group.
Let $C$ be a complete set of representatives of Weyl orbits in $\Gamma^\Z(T)$. Then by \cite[Theorem 1.4.8]{DHL} there is a canonical isomorphism
\[\Grad(BG)\cong\bigsqcup_{\lambda\in C} BL(\lambda),\]
Where $L(\lambda)$ is the centraliser of $\lambda$ in $G$. Therefore $\pi_0(\Grad(\cX))=\Gamma^\Z(T)/W$, and a complex linear form on  $BG$ is given by a map $Z\colon \Gamma^\Z(T)/W\to \C$. The linearity condition in Definition \ref{definition: stack of graded points} amounts to the composition
\[\begin{tikzcd}
\Gamma^\Z(T)\arrow[r] & \Gamma^\Z(T)/W \arrow[r,"Z"] & \C
\end{tikzcd}\]
being a homomorphism. Thus we have an isomorphism
\[\Charges(BG)=\Hom(\Gamma^\Z(T),\C)^W=\C\otimes_\Z \Gamma_\Z(T)^W.\]
The result follows from the fact that the natural map $\Q\otimes\Gamma_\Z(G)\to\Q\otimes_\Z\Gamma(T)_\Z^W$ is an isomorphism \cite{notecharacters}. \end{proof}




Let us now denote $\cP=\Grad(\cX)(\C)$ the groupoid of $\C$-points of $\Grad(\cX)$, and $\cS$ the set of equivalence classes of $\cP$, namely $$\cS = \pi_0(\cP).$$ Thus $\cP$ has objects $(x,\lambda)$, where $x\colon \Spec \C\to \cX$ is a point and $\lambda\colon \Gbb_m\to G_x$ is a cocharacter of the stabiliser group $G_x$ of $x$. A map $(x,\lambda)\to (y,\mu)$ in $\cP$ is an isomorphism $g\colon x\to y$ such that $\mu=\lambda^g$, where $(-)^g\colon G_x\to G_y$ denotes the isomorphism that $g$ induces on automorphism groups by conjugation. From this we see that the set $\cS$ defined here coincides with that in Section \ref{stacks}. We now compare Definition \ref{definition: central charge stack main body} with Definition \ref{definition: central charge on a stack}.

\begin{lemma}\label{lemma: locally constant maps from Grad}
The data of a locally constant map $\abs{\Grad(\cX)}\to \C$ is equivalent to that of a map $Z\colon \cS\to \C$ satisfying that, for every connected finite type scheme $T$ over $\C$ and every map $B\Gbb_m\times Z\to \cX$, the composition
\[\begin{tikzcd}[ampersand replacement=\&]
    {T(\C)} \& \cS \& \C
    \arrow["{Z}", from=1-2, to=1-3]
    \arrow[from=1-1, to=1-2]
\end{tikzcd}\]
is constant.
\end{lemma}
This is precisely the equivariant constancy condition of Definition \ref{definition: central charge stack main body}. 
\begin{proof}
Under the natural injection $\cS\to \abs{\Grad(\cX)}$, $\cS$ is the set of points of $\abs{\Grad(\cX)}$ that can be realised by a map $\Spec \C\to \Grad(\cX)$, and thus $\cS$ inherits a topology from $\abs{\Grad(\cX)}$. 

For any closed subset $R$ of $\abs{\Grad(\cX)}$, the intersection $\cS\cap R$ is dense in $R$, so $\cS$ and $\abs{\Grad(\cX)}$ have the same connected components. For any morphism $T\to \Grad(\cX)$, where $T$ is a scheme over $\C$, the induced map $T(\C)\to \cS$ is continuous. Moreover, if $Y\to \Grad(\cX)$ is a smooth atlas, then the induced map $Y(\C)\to \cS$ is a submersion. Therefore, giving a locally constant map $\abs{\Grad(\cX)}\to \C$ is equivalent to giving a locally constant map $\cS\to \C$, which is in turn equivalent to giving a map $\cS\to \C$ such that, for every morphism $T\to \Grad(\cX)$ with $T$ a scheme of finite type over $\C$ and connected, the composition $T(\C)\to \cS\to \C$ is constant.
\end{proof}

\begin{proposition}\label{proposition: different ways to check linearity}
Let $Z\colon \abs{\Grad(\cX)}\to \C$ be a locally constant map. For a point $x\in \cX(\C)$, with stabiliser group $G_x$, we denote $\psi_x\colon \Gamma^\Z(G_x)\to \C$ the map induced by $Z$ and $x$. Then the following conditions are equivalent: 
\begin{enumerate}[(i)]
\item \label{item 0} The map $Z$ defines a complex linear form on $\cX$.
\item \label{item 1}For every $x\in \cX(\C)$, the map $\psi_x$ is induced by a (uniquely determined) complex character $\chi\in \C\otimes_\Z \Gamma_\Z(G_x)$ of $G_x$.
\item\label{item 2} For every $x\in \cX(\C)$, if $K_x$ is a maximal compact subgroup of $G_x$ and $\mathfrak{k}_x$ is the Lie algebra of $K_x$, then the map $\psi_x$ is induced by a (uniquely determined) complex Lie algebra character $\mathfrak k_x\to \C$.
\item\label{item 3} For every $x\in \cX(\C)$ and for all commuting cocharacters $\lambda,\lambda'\in \Gamma^\Z(G_x)$ we have $\psi_x(\lambda\lambda')=\psi_x(\lambda)+\psi_x(\lambda')$.
\end{enumerate}
\end{proposition}
\begin{proof}
If $Z$ is a  complex linear form on $\X$ and $x\in \cX(\C)$, then there is an induced monomorphism $\iota\colon BG_x\to \cX$ where $G_x$ is the stabiliser group of $x$. The pullback $\iota^*Z$, that is, the composition
\[\begin{tikzcd}[ampersand replacement=\&,column sep=large]
    {\abs{\Grad(BG_x)}} \& {\abs{\Grad(\cX)}} \& \C,
    \arrow["{\abs{\Grad(\iota)}}", from=1-1, to=1-2]
    \arrow["Z", from=1-2, to=1-3]
\end{tikzcd}\]
is a complex linear form on  $BG_x$. It follows from Lemma \ref{lemma: central charges for BG} that $\psi_x$ is given by a complex character $\chi\in \C\otimes_\Z \Gamma_\Z(G_x)$ that is uniquely determined. Therefore (\ref{item 0}) implies (\ref{item 1}).

Now fix $x\in \cX(\C)$. Let $U$ be the unipotent radical of $G_x$ and $L=G_x/U$, which is reductive. Let $\fl$ be the Lie algebra of $L$. Then for any maximal compact subgroup $K_x$ of $G_x$, the composition $K_x\to G_x\to L$ exhibits $L$ as the complexification of $K_x$. Therefore, the Lie algebra $\fl$ of $L$ equals $\C\otimes_\R \mathfrak k_x$, and thus a homomorphism $\mathfrak k_x\to \C$ of real Lie algebras is the same data as a homomorphism $\mathfrak l \to \C$ of complex Lie algebras. Any character $G_x\to \Gbb_m$ factors through $G_x\to L$ and, taking the differential of the induced $L\to \Gbb_m$, it gives a Lie algebra character $\mathfrak l\to \C$. This gives a homomorphism $\Gamma_\Z(G_x)\to \Hom_{\text{Lie}}(\mathfrak l,\C)$ and, by extending scalars, a map $r\colon \C\otimes_Z\Gamma_\Z(G_x)\to \Hom_{\text{Lie}}(\mathfrak l,\C)$. 

An element $\chi \in \C\otimes_Z\Gamma_\Z(G_x)$ gives a pairing map $\langle -,\chi\rangle\colon \Gamma^\Z(G_x)\to \C$. Similarly, a Lie algebra character $\alpha\in \Hom_{\text{Lie}}(\mathfrak l,\C)$ gives a map $\langle -,\alpha\rangle\colon \Gamma^\Z(G_x)\to \C$ as follows. If $\lambda\colon \Gbb_m\to G_x$ is a one-parameter subgroup, then $\langle \lambda, \alpha\rangle$ is the composition of $\Lie(\Gbb_m \xrightarrow{\lambda} G_x\to L)$ and $\alpha$, which is a linear map $\C\to \C$ and thus identified with a complex number. Both pairings are compatible in the sense that $\langle -,\chi\rangle=\langle -,r(\chi)\rangle$ for all $\chi\in \C\otimes_\Z\Gamma_\Z(G_x)$. Therefore, if $\psi_x=\langle -,\chi\rangle$ for some $\chi$, then it also equals $\psi_x=\langle -,r(\chi)\rangle$ and it is thus induced by a Lie algebra character $r(\chi)\colon \mathfrak l=\C\otimes_\R\mathfrak k_x\to \C$, which is uniquely determined because the image of the map $\Gamma^\Z(L)\to \mathfrak l\colon \lambda\to \Lie(\lambda)(1)$ spans $\mathfrak l$ by reductivity of $L$.
This shows that (\ref{item 1}) implies (\ref{item 2}).

Any Lie algebra character respects addition, so it is clear that (\ref{item 2}) implies (\ref{item 3}).

If $g\colon B\Gbb_m^n\to \cX$ is a map, then there is a point $x\in \cX(\C)$ such that $g$ factors as $B\Gbb_m^n\to BG_x\to \cX$. Thus the map $Z_g\colon \Gamma^\Z(\Gbb_m^n)\to \C$ induced by $g$ and $Z$ factors through the map $\psi_x\colon \Gamma^\Z(G_x)\to \C$ induced by $Z$ and $x$. Therefore $Z_g$ is additive for all $g$ if (\ref{item 3}) is satisfied for all $x$, and thus (\ref{item 3}) implies (\ref{item 0}).\end{proof}

\begin{remark}\label{equivalence of definitions of central charge}
Together, Lemma \ref{lemma: locally constant maps from Grad} and Proposition \ref{proposition: different ways to check linearity} establish that Definitions \ref{definition: central charge stack main body} and \ref{definition: central charge on a stack} are equivalent.
\end{remark}

\begin{remark}\label{connection}
If $\cX=X/G$ is a quotient stack, then we can describe
\[\cS=\{(x,\lambda)\colon x\in X(\C),\quad \lambda\colon \Gbb_m\to G_x\}/\sim,\]
where $(x,\lambda)\sim(y,\mu)$ if there is $g\in G(\C)$ such that $y=gx$ and $\mu=\lambda^g$.

Therefore, Definition \ref{definition: central charge stack equivariant main body} of a central charge for the $G$-scheme $X$ is equivalent to Definition \ref{definition: central charge on a stack} of a complex linear form on  $X/G$.
\end{remark}

\end{document}